\documentclass[10pt]{studiamnew}
\usepackage{graphicx}
\usepackage{amsmath}
\usepackage{amssymb}
\newtheorem{theorem}{Theorem}[section]

\newtheorem{corollary}[theorem]{Corollary}

\theoremstyle{definition}
\newtheorem{remark}[theorem]{Remark}

\renewcommand{\theequation}{\thesection.\arabic{equation}}
\numberwithin{equation}{section}

\begin{document}
\title{A new proof for the exact values of $\zeta(2k)$ for $k \in \mathbb{N}$ } %and Riemann zeta function at positive even integers}
\author{Costandin Marius}
\address{Technical University of Cluj-Napoca, \\
          	Str. Memorandumului nr. 28, 400114, \\ Cluj-Napoca, Romania}
\email{constandin@mail.utcluj.ro}

\begin{abstract}
We establish a connection between a function and a series representation  using a similar technique with that that Euler used to solve the Basel problem.  Our result concerns a more general series from which one can obtain $\zeta(2k)$ as a limit case. We also are able to prove the well known result expressing $\zeta(2k) $ with Bernoulli numbers as an application. 
\end{abstract}
\maketitle

\section{Introduction}
Euler solved the Basel problem, (the exact value of $\zeta(2)$) by expressing $\frac{\sin(x)}{x}$ as a product of its zeros. The solution was not rigorous at the time, and later (several years) he solved the more general problem (the exact values for $\zeta(2k), k \in \mathbb{N}$) more rigorously but using another approach. However, the idea of factorization was later developed by K. Weierstrass and refined by Hadamard. In the following we give a well known theorem due to Hadamard:

\begin{theorem}[Hadamard factorization theorem]
Let 
\begin{align}
E_n(z) = \begin{cases} 1-z \hspace{2cm} &n = 0\\ (1 - z)\cdot e^{\sum_{i=1}^n \frac{z^i}{i}}  &n > 0\end{cases}
\end{align} and let $f$ be an entire function of order $\rho$, $A = \{a_n \in \mathbb{C}| f(a_n) = 0, a_n \neq 0\}$ and suppose that $f$ has a zero in $z = 0$ of order $m \geq 0$, then:
\begin{align}
f(z) = z^m e^{P(z)} \prod_{i=1}^n E_{[\rho]}\left( \frac{z}{a_n}\right)
\end{align} where $P(z)$ is a polynomial of degree $N \leq \rho$
\end{theorem}

\section{Main results}

Our main theorem is the following:  

\begin{theorem} For all $x \in \mathbb{R}_+ \setminus \{1\}$ one has 
\begin{align}
\frac{w}{(1 - w)^2} - \frac{1}{\log^2(w)} = 2 \cdot \sum_{n=1}^{\infty} \frac{\log(w)^2 - 4 \pi^2 n^2}{(\log(w)^2 + 4\pi^2 n^2)^2}
\end{align}
\end{theorem}

\begin{proof}
Let $w \in \mathbb{R}_+ \setminus \{1\}$ and let $f(z) = e^{z} - w$. It is known that $f$ is entire of order $\rho = 1$. Let 
\begin{align}
A = \{z_n, n\in \mathbb{Z}| f(z_n) = 0\}
\end{align} Then one can assert, using Hadamard  factorization theorem, that:
\begin{align} 
	e^{z} - w = e^{a_0 + z \cdot a_1} \prod_{z_n\in A} \left(1 - \frac{z}{z_n} \right) e^{\frac{z}{z_n}}
\end{align} The zeros of $f$ can be easily found by letting $z = x + i\cdot y$

 \begin{align}
 e^{x+i\cdot y} = w \hspace{0.3cm} \Rightarrow x = \log(w) \hspace{0.3cm} y = 2 \pi n 
 \end{align} $\forall n \in \mathbb{Z}$. Therefore 
 
 \begin{align}
 e^z - w &= e^{a_0 + z\cdot a_1} \prod_{n\in \mathbb{Z}} \left(1 - \frac{z}{\log(w) + i\cdot 2\pi n} \right) e^{\frac{z}{\log(w) + i\cdot 2\pi n}} \nonumber \\ 
 & = e^{a_0 + z\cdot a_1} \prod_{n\in \mathbb{Z}} \left(1 - z\frac{\log(w) - i\cdot 2\pi n}{\log(w)^2 +  4\pi^2 n^2} \right) \cdot e^{z \sum_{n \in \mathbb{Z}} \frac{\log(w) - i\cdot 2\pi n}{\log(w)^2 +  4\pi^2 n^2}} \nonumber \\ 
 &= e^{a_0 + z\cdot a_1} \prod_{n\in \mathbb{Z}} \left(1 - z\frac{\log(w) - i\cdot 2\pi n}{\log(w)^2 +  4\pi^2 n^2} \right) \cdot e^{z \sum_{n =1}^{\infty} \frac{2 \log(w)}{\log(w)^2 +  4\pi^2 n^2} + \frac{z}{\log^2(w)}}
 \end{align} Let 
 \begin{align}
 g(z) &= \prod_{n\in \mathbb{Z}} \left(1 - z\frac{\log(w) - i\cdot 2\pi n}{\log(w)^2 +  4\pi^2 n^2} \right) 
% = \nonumber \\
% & = \left( 1 - z\frac{1}{\log(w)}\right) \cdot \prod_{n = 1}^{\infty} \left( 1 - z \frac{2\log(w)}{\log(w)^2 + 4 \pi^2 n^2} + z^2\frac{1}{\log(w)^2 + 4 \pi^2 n^2}\right) 
 \end{align} and $\frac{1}{z_n} = \frac{\log(w) - i\cdot 2\pi n}{\log(w)^2 + 4 \pi^2 n^2}$ then 
 \begin{align}
 S_1 &= \sum_{n\in \mathbb{Z}} \frac{1}{z_n} = \frac{1}{\log(w)} + \sum_{n =1}^{\infty} \frac{2 \log(w)}{\log(w)^2 +  4\pi^2 n^2} \nonumber \\
 S_2 &= \sum_{-\infty<n<m<\infty} \frac{1}{z_n\cdot z_m} \hspace{0.3cm} S_3 = \sum_{-\infty<n<m<k<\infty} \frac{1}{z_n\cdot z_m\cdot z_k} \hspace{0.3cm} etc
 \end{align} Because $g$ is holomorphic one has
 
 \begin{align}
 g(z) = 1 - S_1 z + S_2 z^2 - S_3 z^3 + \hdots 
 \end{align} Letting $z = 0$ one obtains $$ 1 - w = e^{a_0} $$
 Therefore from the above identities one can obtain:
 \begin{align}
 e^z - w = (1 - w) g(z) \cdot e^{z \cdot (S_1 + a_1)}
 \end{align} Further more, because $S_1 < \infty$ 
 \begin{align}
 e^{z (S_1 + a_1)} = 1 + \frac{S_1+a_1}{1!} z + \frac{(S_1 + a_1)^2}{2!} z^2 + \hdots
 \end{align} hence 
 
 \begin{align}
 &e^z - w = \nonumber \\
 & = (1 - w) \cdot \nonumber \\
 & \cdot \left( 1 - S_1 z + S_2 z^2 - S_3 z^3 + \hdots  \right) \cdot \left( 1 + \frac{S_1+a_1}{1!} z + \frac{(S_1 + a_1)^2}{2!} z^2 + \hdots \right) \nonumber \\
 & = (1 - w) \cdot \nonumber\\ 
 & \cdot \left( 1 + \left(S_1 + a_1 - S_1 \right)\cdot z + \left( \frac{(S_1 + a_1)^2}{2} -S_1(S_1 + a_1) + S_2\right)z^2 + \hdots \right) \nonumber \\
 & = 1 - w + \frac{1}{1!}z + \frac{1}{2!} z^2 + \hdots 
 \end{align} therefore 
 
 \begin{align}
 (1 - w) a_1 = 1 \hspace{1cm}\Rightarrow \hspace{1cm} a_1 = \frac{1}{1 - w}
 \end{align} and 
 
 \begin{align}
 (1 - w) \left( -\frac{S_1^2}{2} + \frac{a_1^2}{2} + S_2\right) = \frac{1}{2} \hspace{0.5cm} \Rightarrow \hspace{0.5cm} S_1^2 - 2 S_2 = \frac{w}{(1 - w)^2}
 \end{align} but 
 
 \begin{align}
 S_1^2 - 2\cdot S_2 = \sum_{n \in \mathbb{Z}} \frac{1}{z_n^2} = \frac{1}{\log(w)^2} + 2\cdot \sum_{n = 1}^{\infty} \frac{\log(w)^2 - 4 \pi^2 n^2}{(\log(w)^2 + 4\pi^2 n^2)^2}
 \end{align} therefore
 
 \begin{align}
 \frac{w}{(1 - w)^2} - \frac{1}{\log^2(w)} = 2 \cdot \sum_{n=1}^{\infty} \frac{\log(w)^2 - 4 \pi^2 n^2}{(\log(w)^2 + 4\pi^2 n^2)^2}
 \end{align}
 
\end{proof}

%\begin{remark}
%One can see that 
%\begin{align}
%\lim_{w\to 1} 2 \cdot \sum_{n=1}^{\infty} \frac{\log(w)^2 - 4 \pi^2 n^2}{(\log(w)^2 + 4\pi^2 n^2)^2} = - \frac{8 \pi^2}{16 \pi^4} \zeta(2) = -\frac{1}{2 \pi^2} \zeta(2) 
%\end{align} whereas (according to  $Wolfram Alpha ^{\tiny{\textregistered}}$)
%
%\begin{align}
% \lim_{w \to 1} \frac{w}{(1 - w)^2} - \frac{1}{\log^2(w)} = -\frac{1}{12}
%\end{align} hence this confirms the well known result due to Euler, that 
%\begin{align}
%\zeta(2) = \frac{\pi^2}{6}
%\end{align}
%
%\end{remark}

\begin{remark}
Using the holomorphic functions identity one can extend the result to $w \in \mathbb{C}\setminus (-\infty,0]$
\end{remark} 

A simple corollary is the following:

 \begin{corollary}
 For all $x \in \mathbb{C} \setminus \{ i\cdot n | n \in \mathbb{Z}$\} one has:
 
 \begin{align}
 \frac{e^{2\pi x}}{\left(1 - e^{2\pi x}\right)^2} - \frac{1}{4 \pi^2 x^2} = \frac{1}{2\pi^2} \sum_{n=1}^{\infty} \frac{x^2 - n^2}{(x^2+n^2)^2}
 \end{align}
 \end{corollary} 
 
 \begin{proof}
 In the above result let $w = e^{2\pi x}$ then use holomorphic function identity theorem.
 \end{proof}
 
 \begin{remark} One can use Riemann theorem on removable singularities to extent the function on $\mathbb{C}$
 \end{remark}
 \section{Application: Formula for $\zeta(2k)$ with $k \in \mathbb{N}$}
The following theorems allow the computation of $\zeta(2k)$ for $k 
\in \mathbb{N}$. The result is well known due to Euler, we give an alternative proof. 

%\begin{align}
%\frac{d^n h(g(x))}{dx^n} = \sum_{k=1}^n h^{(k)}(g(x))\cdot B_{n,k}(g'(x),g''(x), \hdots, g^{(n-k+1)}(x))
%\end{align} where $B_{n,k}(x_1, \hdots, x_{n-k+1})$ are the Bell polynomials.  

%The following theorem allows the direct computation of $\zeta(2k)$ for $k \in \mathbb{N}$

\begin{theorem}
For $m \in \mathbb{N}$ one has
\begin{align}
(-1)^{m+1} \frac{B_{2m} \cdot (2\pi)^{2m}}{2 (2m)!} = \zeta(2m)
\end{align}

%\begin{align}
%%\lim_{x\to 0} \frac{\frac{d^k f(x)}{d x^k}}{x^k} =
% \lim_{x \to 0} \frac{d^{2k} f(x)}{d x^{2k}}= \frac{\beta_k 2^k k! }{2\pi^2} \zeta(2k+2)
%\end{align} where $$\beta_{k+1} = (-1)^{k} - (k+2) \beta_k $$ and $\beta_0 = -1$
\end{theorem}

\begin{proof} 
Let $$f(x) =  \frac{e^{2\pi x}}{\left(1 - e^{2\pi x}\right)^2} - \frac{1}{4 \pi^2 x^2} $$ and let $h(x,c) = \frac{x - c}{(x + c)^2}$. 
From above one has:

\begin{align}
f(x) = \frac{1}{2\pi^2}\sum_{n=1}^{\infty} h(x^2,n^2)
\end{align} Differentiating $k$ times one obtains 

 \begin{align}\label{E3.4}
 f^{(k)}(x) = \frac{1}{2\pi^2} \sum_{n=1}^{\infty} \frac{\partial^k h(x^2,n^2)}{\partial x^k}
 \end{align}
 
%   Using Faa di Bruno formula, can be stated that :
% 
% \begin{align}\label{E3.5}
% \frac{\partial^k h(g(x),n^2)}{\partial x^k} = \sum_{i=1}^{k-1} h^{(i)}(x^2,n^2) \cdot B_{k,i}(2x,2,\hdots,0) + h^{(k)}(x^2) \cdot B_{n,n}(2x)
%\end{align} where 
The $k$'th derivative of $h(x,c)$ with respect to $x$, can be found as follows

\begin{align}
\frac{1}{x + c} = \sum_{k=0}^{\infty} \frac{(-1)^k}{c^{k+1}}x^k \hspace{0.5cm} \Rightarrow \hspace{0.5cm} \frac{1}{(x + c)^2} = \sum_{k=1}^{\infty} (-1)^{k+1} \frac{k}{c^{k+1}}x^{k-1}
\end{align} therefore 

\begin{align}
\frac{x - c}{(x+c)^2} = \sum_{k=1}^{\infty} (-1)^{k+1} \frac{2k+1}{c^{k+1}} x^k - \frac{1}{c}
\end{align} hence 

\begin{align}
\frac{x^2 - c}{(x^2+c)^2} = \sum_{k=1}^{\infty} (-1)^{k+1} \frac{2k+1}{c^{k+1}} x^{2k} - \frac{1}{c}
\end{align} Differentiating $2m$ times and letting $x \to 0$, one obtains

\begin{align}\label{E3.9}
\lim_{x \to 0} \frac{d^{2m} h(x^2,c)}{dx^{2m}} = (-1)^{m+1} \frac{(2m+1)!}{c^{m+1}}
\end{align}

From Equations (\ref{E3.4}, \ref{E3.9}) one has 

\begin{align}\label{E3.8}
\lim_{x \to 0} \frac{d^{2m} f(x)}{d x^{2m}} &= \frac{(-1)^{m+1} (2m+1)!}{2\pi^2} \sum_{n=1}^{\infty} \frac{1}{n^{2m+2}} \nonumber \\
& = \frac{(-1)^{m+1} (2m+1)!}{2\pi^2} \zeta(2m + 2)
\end{align}

In order to evaluate the left term from the above equation we proceed knowing that $f(z) = \frac{e^{2\pi z}}{(e^{2\pi z} - 1)^2} - \frac{1}{(2\pi z)^2}$, as follows.

\begin{align}
\frac{z}{e^z - 1} = \sum_{k=0}^{\infty} \frac{B_k }{k!} z^k \hspace{0.5cm} \Rightarrow \hspace{0.5cm} \frac{2\pi}{e^{2\pi z}- 1} = \frac{B_0}{z} + \sum_{k=1}^{\infty} \frac{B_k \cdot (2\pi)^k }{k!} z^{k-1}
\end{align} where $B_k$ is the $k$'th Bernoulli number.  Differentiating both sides one obtians:

\begin{align}
-\frac{(2\pi)^2 e^{2\pi z}}{(e^{2\pi z} - 1)^2} = -\frac{B_0}{z^2} + \sum_{k=1}^{\infty} \frac{B_k \cdot (2\pi)^k (k-1) }{k!} z^{k-2}
\end{align} hence 

\begin{align}
\frac{e^{2\pi z}}{(e^{2\pi z}-1)^2} - \frac{1}{(2\pi z)^2} = -\sum_{k = 2}^{\infty} \frac{B_k (2\pi)^{k-2}}{(k-2)! \cdot k} z^{k-2} = -\sum_{k = 0}^{\infty} \frac{B_{k+2} (2\pi)^{k}}{k! \cdot (k+2)} z^{k}
\end{align} Therefore 

\begin{align}\label{E3.12}
\lim_{x\to 0} \frac{d^{2m} }{d x^{2m}} \left( \frac{e^{2\pi z}}{(e^{2\pi z}-1)^2} - \frac{1}{(2\pi z)^2} \right) = - \frac{B_{2m+2} \cdot (2\pi)^{2m}}{2m+2}
\end{align} 

From  Equation (\ref{E3.8}, \ref{E3.12}) one has

\begin{align}
(-1)^{m+2} \frac{B_{2(m+1)} \cdot (2\pi)^{2(m+1)}}{2 (2m+2)!} = \zeta(2m+2)
\end{align}

\end{proof}

%\begin{theorem}
%Let $$f_2(x) =  \frac{e^{2\pi x}}{\left(1 - e^{2\pi x}\right)^2} - \frac{1}{4 \pi^2 x^2} $$ 
%
%\begin{align}
%\zeta(4) = \lim_{x \to 0} f_4(x)
%\end{align}
%\end{theorem}
%
%\begin{proof}
%Indeed form the above equations 
%\begin{align}
%\frac{d f_2(x)}{d x} &= \frac{1}{2\pi^2}\sum_{n=1}^{\infty} \left( \frac{2x}{(x^2 + n^2)^2} - 2 \frac{2x(x^2 - n^2)}{(x^2 + n^2)^3}\right) \nonumber \\
%& = \frac{x}{\pi^2} \sum_{n=1}^{\infty} \frac{-x^2 + 3n^2 }{(x^2 + n^2)^3}
%\end{align} hence 
%\begin{align}
%\lim_{x\to 0} \frac{f_2^{\prime}(x)}{x} = \frac{\pi^2}{30} = \frac{3}{\pi^2} \zeta(4)
%\end{align} hence $\zeta(4) = \frac{\pi^4}{90}$
%\end{proof}

\section{Conclusions}

This paper provides a new proof for the exact values of the Riemann's $\zeta$ function at even natural numbers. In the process, exact values for some series are obtained. The idea belongs to Euler, this paper just uses another function to start with. The values of $\zeta(2k)$ are shown to be particular values of another function.


\begin{thebibliography}{99}
\bibitem{ref2} Boas, R. P., \emph{Entire Functions}, New York: Academic Press Inc. (1954), ch. 2

\bibitem{ref3} Conway, J. B., \emph{Functions of One Complex Variable I}, 2nd ed, New York, Springer-Verlag (1995)

\bibitem{ref1} Knopp, K., \emph{Theory of Functions}, Part II,New York: Dover,  (1996), pp. 1 - 7.


\end{thebibliography}
\end{document}